\documentclass{amsart}

\usepackage{amsmath}
\usepackage{amssymb}
\usepackage{amsthm}
\usepackage{url}
\usepackage{MnSymbol}

\usepackage{hyperref}

\newtheorem{theorem}{Theorem}

\newtheorem{proposition}[theorem]{Proposition}
\newtheorem{claim}[theorem]{Claim}

\theoremstyle{definition}
\newtheorem{definition}[theorem]{Definition}

\theoremstyle{remark}

\def\concat{^\frown}

\begin{document}

\title{The Gamma Question for Many-one Degrees}

\author[M. Harrison-Trainor]{Matthew Harrison-Trainor}
\address{Group in Logic and the Methodology of Science\\
University of California, Berkeley\\
 USA}
\email{matthew.h-t@berkeley.edu}
\urladdr{\href{http://www.math.berkeley.edu/~mattht/index.html}{www.math.berkeley.edu/$\sim$mattht}}


\thanks{The author was partially supported by NSERC PGSD3-454386-2014. The author would like to thank Antonio Montalb\'an, James Walsh, Carl Jockusch, and Andr\'e Nies for their helpful comments.}

\begin{abstract}
A set $A$ is coarsely computable with density $r \in [0,1]$ if there is an algorithm for deciding membership in $A$ which always gives a (possibly incorrect) answer, and which gives a correct answer with density at least $r$. To any Turing degree $\mathbf{a}$ we can assign a value $\Gamma_T(\mathbf{a})$: the minimum, over all sets $A$ in $\mathbf{a}$, of the highest density at which $A$ is coarsely computable. The closer $\Gamma_T(\mathbf{a})$ is to $1$, the closer $\mathbf{a}$ is to being computable. Andrews, Cai, Diamondstone, Jockusch, and Lempp noted that $\Gamma_T$ can take on the values $0$, $1/2$, and $1$, but not any values in strictly between $1/2$ and $1$. They asked whether the value of $\Gamma_T$ can be strictly between $0$ and $1/2$. This is the Gamma question.

Replacing Turing degrees by many-one degrees, we get an analogous question, and the same arguments show that $\Gamma_m$ can take on the values $0$, $1/2$, and $1$, but not any values strictly between $1/2$ and $1$. We will show that for any $r \in [0,1/2]$, there is an $m$-degree $\mathbf{a}$ with $\Gamma_m(\mathbf{a}) = r$. Thus the range of $\Gamma_m$ is $[0,1/2] \cup \{1\}$.

Benoit Monin has recently announced a solution to the Gamma question for Turing degrees. Interestingly, his solution gives the opposite answer: the only possible values of $\Gamma_T$ are $0$, $1/2$, and $1$.
\end{abstract}

\maketitle

\section{Introduction}

We give a solution to the Gamma question for many-one degrees by showing that for each $r \in [0,1/2]$, there is a many-one degree $\mathbf{a}$ such that $\Gamma_m(\mathbf{a}) = r$.

A set $A \subseteq \omega$ is \textit{coarsely computable} if, roughly speaking, we have an algorithm for deciding membership in $A$ which always gives an answer, and the answer is correct except on a set of density zero. By density, we mean asymptotic lower density.

\begin{definition}
The \textit{lower density} of a set $Z \subseteq \omega$ is
\[ \underline{\rho}(Z) := \liminf_{n \to \infty} \frac{|Z \cap [0,n)|}{n}.\]
\end{definition}

\noindent More generally, we can talk about algorithms which are correct half the time, or a third of the time, or almost never. 
To a set $A \subseteq \omega$, we can assign a real number which measures the highest density to which it can be approximated by a computable set.

\begin{definition}[\cite{HJMS}]\label{ccd}
A set $A \subseteq \omega$ is \textit{coarsely computable at density} $r \in [0,1]$ if there is a computable set $R$ such that $\underline{\rho}(A \leftrightarrow R) = r$. Here, $A \leftrightarrow R$ is the set on which $A$ and $R$ agree:
\[ A \leftrightarrow R := \{ x \mid x \in A \Longleftrightarrow x \in R \}.\]
\end{definition}

\begin{definition}[\cite{HJMS}]\label{ccbs}
The \textit{coarse computability bound} of a set $A \subseteq \omega$ is
\[ \gamma(A) := \sup\{r \mid \text{$A$ is coarsely computable at density $r$}\}.\]
That is, $\gamma(A)$ is the supremum, over all computable sets $R$, of $\underline{\rho}(A \leftrightarrow R)$.
\end{definition}

\noindent It is known that for each $r \in (0,1]$, there are sets with coarse computability bound $r$ such that the supremum is obtained, and sets where the supremum is not obtained \cite{HJMS}.

Jockusch and Schupp \cite{JS} have shown that every non-zero Turing degree contains a set which is not coarsely computable. (This follows from the proof of Proposition \ref{prop:1/2} below.) Thus, if $\Gamma_T(\mathbf{a}) = 1$, then $\mathbf{a} = \mathbf{0}$. Andrews, Cai, Diamondstone, Jockusch, and Lempp suggested assigning to each Turing degree a real number which measures the extent to which all sets computable in that degree can be coarsely computed.

\begin{definition}[\cite{ACDJL}]\label{ccbt}
The \textit{coarse computability bound} of a Turing degree $\mathbf{a}$ is
\[ \Gamma_T(\mathbf{a}) := \inf\{\gamma(A) \mid \text{$A$ is $\mathbf{a}$-computable} \}. \]
It suffices to take the infimum only over sets in $\mathbf{a}$.
\end{definition}

Andrews, Cai, Diamondstone, Jockusch, and Lempp showed that $\Gamma_T(\mathbf{a})$ can take on the values $0$, $1/2$, and $1$.\footnote{See also \cite{MN} for a unifying approach to some of these examples.}

\begin{theorem}[\cite{ACDJL}]\label{thm:poss}
For a Turing degree $\mathbf{a}$:
\begin{enumerate}
	\item If $\mathbf{a}$ is computable, $\Gamma_T(\mathbf{a}) = 1$.
	\item If $\mathbf{a}$ is computably traceable and non-computable, $\Gamma_T(\mathbf{a}) = 1/2$.
	\item If $\mathbf{a}$ is 1-random and hyperimmune-free, $\Gamma_T(\mathbf{a}) = 1/2$.
	\item If $\mathbf{a}$ is hyperimmune, $\Gamma_T(\mathbf{a}) = 0$.
	\item If $\mathbf{a}$ is PA, $\Gamma_T(\mathbf{a}) = 0$.
\end{enumerate}
\end{theorem}

Hirschfeldt, Jockusch, McNicholl, and Schupp showed that $\Gamma_T(\mathbf{a})$ cannot take on any values in the open interval $(1/2,1)$. We will repeat the proof here because we will reference it later. 

\begin{proposition}[\cite{HJMS}]\label{prop:1/2}
Let $\mathbf{a}$ be a nonzero Turing degree. Then $\Gamma_T(\mathbf{a}) \leq \frac{1}{2}$.
\end{proposition}
\begin{proof}
Fix $A \in \mathbf{a}$. We will show that there is $B \leq_m A$ such that $\gamma(B) \leq \frac{1}{2}$. The idea is that each bit of $A$ will be copied many times by $B$, so that if we have a computable approximation to $B$ which is correct more than half the time, we can correctly guess at the bits of $A$ with only finitely many errors.

For each $n \in \omega$, define $I_n = [n!,(n+1)!)$. Let
\[ B = \bigcup_{n \in A} I_n.\]
It is easy to see that $B \leq_m A$. Suppose towards a contradiction that $\gamma(B) > \frac{1}{2}$. Let $R$ be a computable approximation to $B$, with $\underline{\rho}(B \leftrightarrow R) > \frac{1}{2}$.
Fix $N$ and $p$ such that for all $n \geq N$,
\[ \frac{|(B \leftrightarrow R) \cap [0,n)|}{n} \geq p > \frac{1}{2}. \]
Increasing $N$, we may assume that $\frac{1}{2} + \frac{1}{N} < p$.

Given $n \geq N$, we will show how to decide computably whether $n \in A$. We claim that $n \in A$ if and only if more than half of the elements of $I_n$ are in $R$. Indeed, suppose that $n \in A$, but at most half of the elements of $I_n$ are in $R$. Then
\[ \frac{|(B \leftrightarrow R) \cap [0,(n+1)!)|}{(n+1)!} \leq \frac{n! + \frac{(n+1)!-n!}{2}}{(n+1)!} = \frac{1}{2} + \frac{1}{2(n+1)} < p.\]
This is a contradiction. So if $n \in A$, more than half of the elements of $I_n$ are in $R$. A similar argument works when $n \notin A$.
\end{proof}

The Gamma question, from \cite{ACDJL}, asks whether the value of $\Gamma_T$ can be strictly between $0$ and $1/2$. Monin \cite{M} has recently given a solution to the Gamma question: The only possible values of $\Gamma_T$ are $0$, $1/2$, and $1$.

Our work grew out of an independent attempt to answer the Gamma question. If we replace Turing reducibility by many-one reducibility, we get a Gamma function on many-one degrees:
\begin{definition}\label{ccbm}
The \textit{coarse computability bound} of an $m$-degree $\mathbf{a}$ is
\[ \Gamma_m(\mathbf{a}) := \inf\{\gamma(A) \mid \text{$A \leq_m \mathbf{a}$} \}. \]
It suffices to take the infimum only over sets in $\mathbf{a}$.
\end{definition}
The proof of Proposition \ref{prop:1/2} used a many-one reduction, so it still holds for $m$-degrees. Moreover, the examples in Theorem \ref{thm:poss} yield examples of $m$-degrees with $\Gamma_m$ being $0$, $1/2$, and $1$. Thus, we can ask the Gamma question for $m$-degrees: Can the value of $\Gamma_m$ be strictly between $0$ and $1/2$? Interestingly, we get the opposite answer from Monin's: Every $p \in [0,1/2]$ is a possible value of $\Gamma_m$.

\begin{theorem}\label{thm:main}
Fix $0 \leq p \leq \frac{1}{2}$. There is an $m$-degree $\mathbf{a}$ with $\Gamma_m(\mathbf{a}) = p$.
\end{theorem}

Versions of the Gamma question for weaker reducibilities have already been asked in the literature: In \cite{H}, Hirschfeldt asked the Gamma question for truth table degrees. (Monin's answer to the Gamma question for Turing degrees also yields the same answer for truth table degrees: The value of $\Gamma_{tt}$ cannot be strictly between $0$ and $1/2$.) An interesting question is what happens for intermediate reductions, such as bounded truth table reductions. Do such reductions have enough computational power to apply the theorems from coding theory used by Monin, or are they sufficiently simple to allow a construction such as the one we use for many-one degrees?

\section{Background on the Hypergeometric Distribution}

The proof of Theorem \ref{thm:main} will make use of a probabilistic argument about random variables following a hypergeometric distribution. We will quickly review this distribution here. (See \cite[p. 52]{Hoel}.)

The hypergeometric distribution is the discrete probability distribution of the number of successes in $K$ draws, without replacement, from a population of size $N$ which contains $n$ successes. For example, one might think of red and blue marbles in a box; if there are $N$ marbles, $n$ of which are red, and we randomly select $K$ marbles, the number of red marbles we pick will follow a hypergeometric distribution. We denote the hypergeometric distribution by $H(K,N,n)$ and, if $X \sim H(K,N,n)$, we have
\[Pr(X = x) = \frac{\binom{n}{x}\binom{N-n}{K-x}}{\binom{N}{K}}. \]
Our particular application will be to have a set $U$ of size $N$, with a subset $V$ of size $n$. If $p > q$ are real numbers in $[0,1]$, we will randomly pick from $U$ a set $S$ consisting of about $pN$ elements. We want to choose, in $S$, at least $qn$ elements which are also in $V$. Since $p > q$, it is reasonable to think that we should often get enough elements of $V$. Intuitively, the larger $N$ and $n$ are, the more likely we are to get want we want. Precise bounds are given by the following theorem:

\begin{theorem}[\cite{Hoeffding}, see also \cite{Chvtal}]
Let $X \sim H(K,N,n)$ where $p = K / N > q$. Let $t = p-q$. Then
\[ Pr(X \leq qn) \leq \exp(-2t^2n).\] 
\end{theorem}

\noindent It will be important that this bound does not depend on $N$ (though of course, as $n$ becomes bigger, $N$ will as well).

\section{Proof of the Main Theorem}

We will now prove Theorem \ref{thm:main}.

\begin{proof}[Proof of Theorem \ref{thm:main}]
Fix $0 < p < \frac{1}{2}$. We will find a set $A$ whose $m$-degree $\mathbf{a}$ has $\Gamma_m(\mathbf{a}) = p$. Fix $(C_\ell)_{\ell\in \omega}$ a non-effective list of the computable sets, in which each set is repeated infinitely many times.

We will ensure that $\gamma(A) \leq p$ having, for each computable set $C_\ell$,
\[ \liminf_{n \to 0} \frac{|(A \leftrightarrow C_\ell) \cap [0,n)|}{n} \leq p.\]
We will accomplish this by making sure that for each $\ell$, there are infinitely many values of $n$ for which we force $A$ to differ from $C_\ell$ on a large portion of $[0,n)$. This will force $\Gamma_m(\mathbf{a}) \leq p$. In fact, since each computable set appears infinitely many times in the list $(C_\ell)_{\ell \in \omega}$, it suffices to find, for each $\ell$, a single $n \geq \ell$ with 
\[ \frac{|(A \leftrightarrow C_\ell) \cap [0,n)|}{n} \leq p + \epsilon_\ell \]
where $\epsilon_\ell \to 0$.

To have $\Gamma_m(\mathbf{a}) \geq p$, we must make sure that for each set $B$ which is $m$-reducible to $A$ via $f$, $\gamma(B) \geq p$. (One such $B$ will be $A$ itself via the identity reduction.) We will think of $A$ as being approximated by $A^* = \varnothing$. (So we want the bits of $A$ to be $0$'s with density at least $p$.) Thus we might initially try to approximate $B$ by $B^* = \varnothing$, which is what we get by applying the reduction $f$ to $A^*$. This will not work, for the reason that $f$ could be highly non-injective. For example, if $f$ maps every element to the same element $y$, then we could have $A = \{y\}$, which is very well approximated by $A^*$, but applying the reduction $f$ we get $B = \omega$ which is very badly approximated by $B^*$. This is where we will exploit the fact that $p \leq 1/2$. Say $z = f(x) = f(y)$. Then if we put $x \in B^*$ but $y \notin B^*$, we are guaranteed to be right about at least one of the two; for if $z \in A$, then $x,y \in B$ in which case we were right about $x$, and if $z \notin A$, then $x,y \notin B$, in which case we were right about $y$. If we manage this in the right way, we will be correct with density $1/2$. (The proof of Proposition \ref{prop:1/2} shows that we must do something like this.)

More formally, let $(f_e \colon \omega \to \omega)_{e \in \omega}$ be a (non-effective) list of the total many-one reductions.
For each $e$, let
\[B_e = f_e^{-1}(A) := \{ x \mid f_e(x) \in A \}.\] 
We define a computable approximation $B_e^*$ to $B_e$ as follows.
First, let $I_1,I_2,I_3,\ldots$ be the consecutive intervals of length one, two, three, and so on. (So $I_1 = \{0\}$, $I_2 = \{1,2\}$, $I_3 = \{3,4,5\}$, etc.)
For each interval $I_n$, let $J_{n,e} = f_e(I_n)$ be the multiset image\footnote{Recall that multisets are a generalization of sets to allow multiple instances of the same element. By the multiset image, we mean that we want to count the number of pre-images of an element in the range.} of $I_n$ under $f_e$, and write $J_{n,e} = J_{n,e}^{*} \uplus 2 J_{n,e}^{**}$ where each element has multiplicity one in $J_{n,e}^{*}$. So, for example, if $f_e(3)=0$, $f_e(4) = 0$, and $f_e(5) = 1$, then $f_e(I_3) = \{0,0,1\} = \{1\} \uplus 2\{0\}$; if $J_{8,e} = f_e(I_8) = \{0,0,0,0,0,1,1,2\}$, then $J_{8,e} = \{0,2\} \uplus 2\{0,0,1\}$. We can write $I_n$ as a disjoint union $I_{n,e}^{*} \cup I_{n,e}^{**,1} \cup I_{n,e}^{**,2}$, where $f_e(I_{n,e}^{**,1}) = f_e(I_{n,e}^{**,2}) = J_{n,e}^{**}$ and $f_e(I_{n,e}^{*}) = J_{n,e}^{*}$. Then let $B^* = \bigcup_{n} I_{n,e}^{**,1}$. The simplicity with which we can describe $B^*$ is where we take advantage of the fact that we are considering many-one reductions rather than Turing reductions. We will have that, for some decreasing positive sequence $\gamma_n \to 0$, and for all $n$,
\begin{equation*} \frac{|(B_e \leftrightarrow B_e^*) \cap I_n|}{n} \geq p - \gamma_n. \tag{$\dagger$}\label{assum}\end{equation*}
Since the length of the intervals $I_n$ are increasing slowly, this will suffice to get $\gamma(B_e) \geq p$.

\begin{claim}\label{clm1}
Assuming \eqref{assum}, $\gamma(B_e) \geq p$.
\end{claim}
\begin{proof}
Since
\[ \frac{|(B_e \leftrightarrow B_e^*) \cap I_m|}{m} \geq p - \gamma_m,\]
for each $m$, there is $K_m$ such that for all $K \geq K_m$,
\[ \frac{|(B_e \leftrightarrow B_e^*) \cap (\bigcup_{n \leq K} I_n)|}{\sum_{n \leq K} n} \geq p - \gamma_m.\]
(Assume that the sequence $K_m$ is strictly increasing in $m$.)
Given $m$, let $x \in I_{N+1}$ for some $N \geq K_m$, with $N$ sufficiently large that $\frac{N-1}{N+1} (p - \gamma_m) \geq p - 2\gamma_m$. Then
\[ |(B_e \leftrightarrow B_e^*) \cap (\bigcup_{n < N} I_n)| \leq |(B_e \leftrightarrow B_e^*) \cap [0,x]|. \]
So
\[ \frac{|(B_e \leftrightarrow B_e^*) \cap [0,x]|}{x+1} \geq \frac{(\sum_{n < N} n)}{x+1} \frac{|(B_e \leftrightarrow B_e^*) \cap (\bigcup_{n < N} I_n)|}{(\sum_{n < N} n)} \geq \frac{(\sum_{n < N} n)}{x+1}(p - \gamma_m).\]
Now
\[ \frac{(\sum_{n < N} n)}{x+1} \geq \frac{(\sum_{n < N} n)}{(\sum_{n \leq N} n)} = \frac{N-1}{N+1}.\]
By choice of $N$, $\frac{N-1}{N+1} (p - \gamma_m) \geq p - 2\gamma_m$, and so
\[ \frac{|(B_e \leftrightarrow B_e^*) \cap [0,x]|}{x+1} \geq p - 2\gamma_m.\]
Thus
\[ \gamma(B_e) = \liminf_{x \to \infty} \frac{|(B_e \leftrightarrow B_e^*) \cap [0,x)|}{x} \geq p.\qedhere\]
\end{proof}

Note that one of the reductions $f_e$ is the identity reduction, so for that $e$, $B_e = A$. Thus $\gamma(A) = p$. Hence $\Gamma_m(A) = p$.

\medskip{}

We now construct $A$ as the concatenation of infinitely many finite binary strings: $A = \alpha_1 \concat \beta_1 \concat \alpha_2 \concat \beta_2 \concat \alpha_3 \concat \dots$. Each $\alpha_i$ will consist entirely of $0$'s.\footnote{We should expect to have long sequences of zeros. Since $\gamma(A) = p \leq 1/2$, $A \leftrightarrow \omega$ should have density at most $p$. But that means that the upper destiny of $A \leftrightarrow \varnothing$ should be at least $1 - p \geq 1/2$, so there should be initial segments of $A$ with many $0$'s.} Note that we have already defined each of the approximations $B_e^*$ computably; our construction of $A$ can be non-effective. Fix $(\epsilon_\ell)_{\ell \in \omega}$ a decreasing sequence of positive reals converging to zero (and assume that the $\epsilon_\ell$ are small relative to $p$, so that for example $p-\epsilon_\ell > 0$ and $p + \epsilon_\ell < 1/2$). Given $\alpha_1 \concat \beta_1 \concat \alpha_2 \concat \beta_2 \concat \cdots \concat \beta_\ell$, we must define $\alpha_{\ell+1}$ and $\beta_{\ell + 1}$. Let $L$ be the length of $\alpha_1 \concat \beta_1 \concat \alpha_2 \concat \beta_2 \concat \cdots \concat \beta_\ell$.

At stage $\ell$, we consider only the reductions $f_e$ for $e \leq \ell$. We have two competing desires. First, we want to define $\beta_{\ell+1}$ so that $A$ has very little agreement with $C_\ell$ on this part of their domain. Second, we want $\beta_{\ell+1}$ to have many $0$'s, particularly on elements in the ranges of the $f_e$, so that the sets $B_e$ have many $0$'s. Our probabilistic argument using the hypergeometric distribution will show that we can satisfy both of these desires at the same time.

At stage $\ell$, we will divide the intervals $I_n$ which make up the domain of $B_e$ into three types, depending on their size: small, medium, and large. An interval will be of medium size for exactly one stage $\ell$; at earlier stages, it will be large, and at later stages, it will be small. Small intervals are too small to use the bounds for the hypergeometric distribution while keeping any errors under $\epsilon_\ell$; we will ensure that their images, under the $f_e$, look only at the values in $\alpha_{\ell + 1}$ and earlier. The medium and large intervals are large enough to use the bounds for the hypergeometric distribution. The difference between the medium and large intervals will not show up until the verification. Essentially, large intervals are so big that the values of $\beta_{\ell+1}$ do not affect the agreement of $B_e$ with $B_e^*$.

Choose $M$ sufficiently large (and bigger than the value of $M$ at the previous stage) so that:
\begin{enumerate}
	\item $\ell \sum_{n \geq M} \exp(-\epsilon_\ell^3 n) < 1$ and
	\item if $M' \geq M$ and $M' = 2m_1 + m_2$, then $\frac{m_1 + pm_2 - L}{M'} \geq p - \epsilon_\ell$.
\end{enumerate}
Since $p \leq 1/2$, (2) holds for $M \geq L / \epsilon_\ell$. We can get (1) to hold for sufficiently large $M$ because the infinite series
$\ell \sum_{n \geq 1} \exp(-\epsilon_\ell^3 n)$ converges. $M$ is the cutoff between the small and medium intervals. (1) says that $M$ is big enough to apply the tail bounds for the hypergeometric distribution, and (2) says that $M$ is large enough compared to $L$ that we do not have to worry about what we have already chosen for $\alpha_1 \concat \beta_1 \concat \alpha_2 \concat \beta_2 \concat \cdots \concat \beta_\ell$.

Choose $K$ sufficiently large so that each element of $J_{i,e}$, for $i < M$ and $e \leq \ell$, is less than $L + K$. Set $\alpha_{\ell+1} = 0^K$. Thus every small interval $I_{i,e}$ will see, under the reduction $f_e$, only $\alpha_1 \concat \beta_1 \concat \alpha_2 \concat \beta_2 \concat \cdots \concat \beta_\ell \concat \alpha_{\ell + 1}$. Our choice of $\beta_{\ell+1}$ will not affect the reduction on these intervals.

Then choose $N$ sufficiently large that so that $\frac{L + K + pN + \epsilon_\ell N}{L+K+N} \leq p + 2\epsilon_\ell$ and $\frac{\lfloor (p + \epsilon_\ell) N \rfloor}{N} - p \geq \epsilon_\ell / 2$. We will have $|\beta_{\ell + 1}| = N$. $N$ is large enough that, by making $\beta_{\ell+1}$ agree with $C_\ell$ at density at most $p$, we can make $A$ agree with $C_\ell$ at density at most $p + \epsilon_\ell$. This $N$ will be the $N$ in our applications of the tail bounds for the hypergeometric distribution. Because $K$ was chosen dependent on $M$, and $N$ was chosen dependent on $K$, it is important here that these tail bounds depend only on $n$ (which will be bounded below by a fixed multiple of $M$). 

As before, write $J_{n,e} = f_e(I_n)$ as $J_{n,e}^{*} \uplus 2 J_{n,e}^{**}$ where each element has multiplicity one in $J_{n,e}^*$. Recall that, by choice of $B_e^*$, we are guaranteed to have $B_e$ agree with $B_e^*$ on half of the elements in $I_{n,e}$ which map to $J_{n,e}^{**}$. If the size of $J_{n,e}^{**}$ is very close to the size of $J_{n,e}$, then we are guaranteed to get agreement of close to $1/2$ on $I_{n,e}$, and we do not have to worry about $I_{n,e}$. The other option is that $J_{n,e}^{*}$ is very large, in which case we get small tail bounds for the hypergeometric distribution. More formally, if $n \geq M$, then either $\frac{|J_{n,e}^{**}|}{|J_{n,e}|} \geq \frac{1}{2} - \epsilon_\ell$ or $|J_{n,e}^*| > 2 \epsilon_\ell |J_{n,e}|$.
Indeed, if we are not in the former case, then
\[ 2 |J_{n,e}^{**}| < |J_{n,e}| - 2 \epsilon_\ell |J_{n,e}|.\]
Rearranging, and using the fact that $|J_{n,e}| = |J_{n,e}^*| + 2 |J_{n,e}^{**}|$, we get
\[ |J_{n,e}^*| > 2 \epsilon_\ell |J_{n,e}|.\]
Let $\Omega$ index the pairs $n \geq M$ and $e \leq \ell$ for which $|J_{n,e}^*| > 2 \epsilon_\ell |J_{n,e}|$.

\begin{claim}
There is a set $S \subseteq [L+K,L+K+N)$ with $\frac{|S|}{N} \leq p + \epsilon_\ell$ such that for each $(n,e) \in \Omega$,
\[ \frac{|S \cap J_{n,e}^{*}| + |J_{n,e}^{*} \setminus [L+K,L+K+N)|}{|J_{n,e}^{*}|} \geq p. \]
\end{claim}

The set $S$ is a set on which $A$ will be forced to have $0$'s. On the other elements of $[L+K,L+K+N)$, $A$ will have the freedom to be different from $C_\ell$. This claim says that we can choose $S$ to simultaneously have $S$ small enough that $A$ can be sufficiently different from $C_\ell$ and large enough that the reductions $f_e$ find sufficiently many $0$'s in their ranges.

\begin{proof}
First, note that if we modify $J_{n,e}^{*}$ by removing an element which is outside of the interval $[L+K,L+K+N)$ and adding a new element which is inside of this interval, for any fixed set $S \subseteq [L+K,L+K+N)$ the quantity
\[ \frac{|S \cap J_{n,e}^{*}| + |J_{n,e}^{*} \setminus [L+K,L+K+N)|}{|J_{n,e}^{*}|} \]
can only decrease. Also, if $J_{n,e}^{*} \supseteq [L+K,L+K+N)$, then for any choice of $S$ with $\frac{|S|}{N} \geq p$ we will have
\[ \frac{|S \cap J_{n,e}^{*}| + |J_{n,e}^{*} \setminus [L+K,L+K+N)|}{|J_{n,e}^{*}|} \geq p \]
as desired. So we may assume that, for each $(n,e)$, $J_{n,e}^* \subseteq [L+K,L+K+N)$.

We give a probabilistic argument that the desired set $S$ exists. Imagine that we randomly pick a set $S$ of size $r = \lfloor (p + \epsilon_\ell) N \rfloor$. For each $(n,e) \in \Omega$, let $X_{n,e}$ be the random variable $|S \cap J_{n,e}^*|$; we have $X_{n,e} \sim H(r,N,|J_{n,e}^*|)$. Let $t = \frac{r}{N} - p$. By choice of $N$, we have $t \geq \epsilon_\ell / 2$. So by the tail bounds for the hypergeometric distribution, the probability that for some fixed $(e,n) \in \Omega$, $|S \cap J_{n,e}^*| \leq p |J_{n,e}^*|$ is bounded above by
\[ Pr(X_{e,n} \leq p |J_{n,e}^*|) \leq \exp(-2 t^2 |J_{n,e}^*|) \leq \exp(-\epsilon_\ell^3 |J_{n,e}|) = \exp(-\epsilon_\ell^3 n).\]
(Note that this holds even if $p |J_{n,e}^*|$ is not an integer.) So the probability that for all $(e,n) \in \Omega$, $|S \cap J_{n,e}^*| \leq p |J_{n,e}^*|$, is bounded above by
\[ \sum_{(e,n) \in \Omega} \exp(-\epsilon_\ell^3 n) \leq \ell \sum_{n \geq M} \exp(-\epsilon_\ell^3 n) < 1.\]
So there is a non-zero probability that we pick a set $S$ as desired; some such set must exist.
\end{proof}

For $i < N$, when $L + K + i \in S$, set $\beta_{\ell+1}(i) = 0$, and otherwise set $\beta_{\ell+1}(i) \neq C_\ell(L + K + i)$. So for $x \in [L+K,L+K+N)$, if $x \in S$ then $A(x) = 0$ and if $x \notin S$ then $A(x) \neq C_\ell(x)$.

First, we will show that we made $A$ sufficiently different from $C_\ell$.

\begin{claim}
$\gamma(A) \leq p$.
\end{claim}
\begin{proof}
We have, for each $\ell$, that
\begin{align*}
\frac{|(A \leftrightarrow C_\ell) \cap [0,L+K+N)|}{L+ K + N} \leq& \frac{L + K + | \{x \in S \mid A(i) = C_\ell(i) \}|}{L+ K + N}\\ \leq& \frac{L + K + |S|}{L + K + N}\\ \leq& \frac{L + K + pN + \epsilon_\ell N}{L+K+N}.
\end{align*}
Here, $L$, $K$, and $N$ are the values of those variables at stage $\ell$ of the construction. By choice of $N$,
\[ \frac{L + K + pN + \epsilon_\ell N}{L+K+N} \leq p + 2 \epsilon_\ell.\]
So
\[ \frac{|(A \leftrightarrow C_\ell) \cap [0,L+K+N)|}{L+ K + N} \leq p + 2 \epsilon_\ell.\]
Then, noting that for each $\ell$ there are infinitely many $\ell'$ with $C_\ell = C_{\ell'}$,
\[ \gamma(A) = \liminf_{n \to \infty} \frac{|(A \leftrightarrow C_\ell) \cap [0,n)|}{n} \leq p. \qedhere\]
\end{proof}

Second, we will show that $B_e$ is sufficiently well approximated by $B_e^*$. The following claim verifies the hypotheses of Claim \ref{clm1}.

\begin{claim}
Fix $e$. Given $n$, let $\ell \geq e$ be such that $M_{\ell+1} > n \geq M_{\ell}$, where $M_\ell$ is the value of $M$ at stage $\ell$. Then
\[ \frac{|(B_e \leftrightarrow B_e^*) \cap I_n|}{n} \geq p - \epsilon_\ell.\]
\end{claim}

The $n$ satisfying $M_{\ell+1} > n \geq M_{\ell}$ are the sizes of intervals of medium size for $\ell$, that is, those where $\beta_\ell$ is exactly the right length to determine the amount of agreement between $B_e$ and $B_e^*$ on the interval $I_n$.

\begin{proof}
Write $I_e = I_{n,e}^{*,1} \cup I_{n,e}^{*,2} \cup I_{n,e}^{*,3} \cup I_{n,e}^{**,1} \cup I_{n,e}^{**,2}$ where $I_{n,e}^{**,1}$ and $I_{n,e}^{**,2}$ are as before, and $f_e(I_{n,e}^{*,1}) \subseteq [0,L)$, $f_e(I_{n,e}^{*,2}) \subseteq [L+K,L+K+N)$, and $f_e(I_{n,e}^{*,3}) \subseteq [L,L+K) \cup [L+K+N,\infty)$. (Again, $L$, $K$, and $N$ are the values at stage $\ell$.)

For $x \in I_{n,e}^{*,3}$, $f_e(x) \in [L,L+K) \cup (L+K+N,\infty]$, and so, since $n < M_{\ell+1}$, $B_e^*(x) = 0 = A(f_e(x)) = B_e(x)$. Thus
\[ |(B_e \leftrightarrow B_e^*) \cap I_{n,e}^{*,3}| = |I_{n,e}^{*,3}|.\]
For $I_{n,e}^{*,2}$, if $f(x) \in S$, then $B_e(x) = B_e^*(x)$ since $x \notin B_e^*$ and $f_e(x) \notin A$, so that $x \notin B_e$. Thus (recalling that $J_{n,e}^{*,2} = f_e(I_{n,e}^{*,2})$, and that $f_e$ is injective on this set):
\[ |(B_e \leftrightarrow B_e^*) \cap I_{n,e}^{*,2}| \geq |S \cap J_{n,e}^{*,2}|.\]
By choice of $S$, we have
\[ \frac{|S \cap J_{n,e}^{*}| + |J_{n,e}^{*} \setminus [L+K,L+K+N)|}{|J_{n,e}^{*}|} \geq p. \]
Thus, noting that $|J_{n,e}^{*} \setminus [L+K,L+K+N)| = |I_{n,e}^{*,1}| + |I_{n,e}^{*,3}|$,
\begin{equation*}\label{*} |(B_e \leftrightarrow B_e^*) \cap I_{n,e}^{*}| + L \geq p|I_{n,e}^{*}|. \tag{$*$}\end{equation*}

By definition of $B_e^*$,
\begin{equation*}\label{**} |(B_e \leftrightarrow B_e^*) \cap I_{n,e}^{**}| = |I_{n,e}^{**,1}| = |I_{n,e}^{**,2}|.\tag{$**$}\end{equation*}
This is because each $x \in I_{n,e}^{**,1}$ can be paired with a $y \in I_{n,e}^{**,2}$ with $f_e(x) = f_e(y)$; we have $x \in B_e^*$ and $y \notin B_e^*$. Either $f_e(x) = f_e(y) \in A$, in which case $x,y \in B_e$, or $f_e(x) = f_e(y) \notin A$, in which case $x,y \notin B_e$.

So combining equations \eqref{*} and \eqref{**}, we get
\[ |(B_e \leftrightarrow B_e^*) \cap I_{n,e}| \geq |I_{n,e}^{**,1}| + p|I_{n,e}^{*}| - L.\]
By choice of $M_{\ell}$,
\[\frac{|I_{n,e}^{**,1}| + p |I_{n,e}^{*}| - L}{|I_{n,e}|} \geq p - \epsilon_\ell. \qedhere\]
\end{proof}

This completes the proof of the theorem, as Claim \ref{clm1} now gives that $\gamma(B_e) \geq p$ for each $e$.
\end{proof}

\nocite{HJKS}
\bibliography{References}

\newcommand{\etalchar}[1]{$^{#1}$}
\begin{thebibliography}{ACD{\etalchar{+}}16}

\bibitem[ACD{\etalchar{+}}16]{ACDJL}
Uri Andrews, Mingzhong Cai, David Diamondstone, Carl Jockusch, and Steffen
  Lempp.
\newblock Asymptotic density, computable traceability, and 1-randomness.
\newblock {\em Fundamenta Mathematicae}, 2016.

\bibitem[Chv79]{Chvtal}
Va\v{s}ek Chv{\'a}tal.
\newblock The tail of the hypergeometric distribution.
\newblock {\em Discrete Math.}, 25(3):285--287, 1979.

\bibitem[Hir]{H}
Denis Hirschfeldt.
\newblock Some questions in computable mathematics.
\newblock To appear.

\bibitem[HJKS16]{HJKS}
Denis~R. Hirschfeldt, Carl~G. Jockusch, Rutger Kuyper, and Paul~E. Schupp.
\newblock Coarse reducibility and algorithmic randomness.
\newblock {\em J. Symb. Log.}, 81(3):1028--1046, 2016.

\bibitem[HJMS16]{HJMS}
Denis Hirschfeldt, Carl Jockusch, Timothy McNicholl, and Paul Schupp.
\newblock Asymptotic density and the coarse computability bound.
\newblock {\em Computability}, 5(1):13--27, 2016.

\bibitem[Hoe63]{Hoeffding}
Wassily Hoeffding.
\newblock Probability inequalities for sums of bounded random variables.
\newblock {\em J. Amer. Statist. Assoc.}, 58:13--30, 1963.

\bibitem[HPS71]{Hoel}
Paul Hoel, Sidney Port, and Charles Stone.
\newblock {\em Introduction to probability theory}.
\newblock Houghton Mifflin Co., Boston, Mass., 1971.
\newblock The Houghton Mifflin Series in Statistics.

\bibitem[JS12]{JS}
Carl Jockusch and Paul Schupp.
\newblock Generic computability, {T}uring degrees, and asymptotic density.
\newblock {\em J. Lond. Math. Soc.}, 85(2):472--490, 2012.

\bibitem[MN15]{MN}
Benoit Monin and Andr\'e Nies.
\newblock A unifying approach to the gamma question.
\newblock {\em Proceedings of the Thirtieth Annual IEEE Symposium on Logic in
  Computer Science (LICS 2015)}, pages 585--596, July 2015.

\bibitem[Mon]{M}
Benoit Monin.
\newblock Asymptotic density and error-correcting codes.
\newblock Preprint.

\end{thebibliography}
\bibliographystyle{alpha}

\end{document}